\documentclass[reqno,centertags, 12pt]{amsart}

\usepackage{amssymb,amsmath,amsfonts,amssymb} 
-\marginparwidth \oddsidemargin -9mm \evensidemargin \oddsidemargin
\usepackage{latexsym}
\advance\hoffset by 5mm

\def\@abssec#1{\vspace{.05in}\footnotesize \parindent .2in
{\bf #1. }\ignorespaces}
\newtheorem{theorem}{Theorem}[section]

\newtheorem{lemma}[theorem]{Lemma}
\newtheorem{proposition}[theorem]{Proposition}

\newtheorem{definition}[theorem]{Definition}

\def \Rm {\mathbb R}
\def \Tm {\mathbb T}

\def\cP{\mathcal P}

\allowdisplaybreaks \numberwithin{equation}{section}

\title{Nonlocal maximum principles for active scalars}
\author{Alexander Kiselev}
\thanks{Department of
Mathematics, University of Wisconsin, Madison, WI 53706, USA;
email: kiselev@math.wisc.edu}

\begin{document}


\begin{abstract}
Active scalars appear in many problems of fluid dynamics. The most common examples
of active scalar equations are 2D Euler, Burgers, and 2D surface quasi-geostrophic
equations. Many questions about regularity and properties of solutions of these
equations remain open. We develop the idea of nonlocal maximum principle introduced in \cite{KNV},
formulating a more general criterion and providing new applications. The most interesting
application is finite time regularization of weak solutions in the supercritical regime.
\end{abstract}

\maketitle

\section{Introduction}\label{intro}

We will call the function $\theta(x,t)$ (dissipative) active scalar if it satisfies
\begin{equation}\label{as1}
\theta_t = u \cdot \nabla \theta + (-\Delta)^{\alpha}\theta,  \,\,\,\theta(x,0)=\theta_0(x),
\end{equation}
where $0 \leq \alpha \leq 1,$ and the vector field $u$ is determined  by $\theta.$
In the case $\alpha=0$ we will usually mean that the equation is conservative, so that dissipation term vanishes
(instead of being equal to $\theta$).
In this paper, we will consider \eqref{as1} on a torus $\Tm^d$ (equivalently, we can talk about periodic solutions in $\Rm^d$). 
Active scalars appear in many problems coming from fluid mechanics. The classical examples are
the 2D Euler equation in the vorticity form where $u = \nabla^\perp (-\Delta)^{-1}\theta$ and the Burgers equation
where $u = \theta.$ Another important example is the 2D surface quasi-geostrophic (SQG) equation coming from
atmospheric science, for which we have $u = \nabla^\perp (-\Delta)^{-1/2}\theta.$ More generally, the
scale of active scalars that interpolate between the SQG and 2D Euler is called modified SQG, with $u =
\nabla^\perp (-\Delta)^{-\gamma}\theta,$ $1/2 < \gamma < 1.$ Another example of equation that is closely related
is the porous media equation with fractional dissipation (see e.g. \cite{CCGO}).
See \cite{Const1} for an overview of active scalars and more
background information.

The SQG equation
appeared in the mathematical literature for the first time in \cite{CMT1}, and since then has attracted significant attention.
The equation has $L^\infty$ maximum principle \cite{R,CC1}, which makes the $\alpha =1/2$ dissipation critical.
It has been known since \cite{R} that the equation has global smooth solutions (for appropriate initial data) when $\alpha >1/2.$
The global regularity in the critical case has been settled independently in \cite{KNV} (in the periodic setting) and \cite{CV} (in the whole
space as well as in local setting). A third proof of the same
result was provided recently in \cite{KN1}. All proofs are quite different. The method of \cite{CV} is inspired by Di-Giorgi iterative
estimates. The approach of \cite{KN1} uses appropriate set of test functions and estimates on their evolution.
The method of \cite{KNV}, on the other hand, is based on a new technique which can be called a nonlocal
maximum principle. The idea is to prove that the evolution according to \eqref{as1} preserves a certain modulus of continuity $\omega$ of
solution. The control is strong enough to give uniform bound of $\|\nabla \theta\|_{L^\infty}$ in the critical case, which is
sufficient for global regularity.

The method of \cite{KNV} has been applied to many other problems: Burgers equation \cite{KNS,DDL}, modified SQG equation \cite{MX},
whole space SQG \cite{AH,DD}, porous media equation \cite{CCGO}, and other models \cite{Dong1a,LR1}. 
Our goal here is to generalize this method, making it applicable to wider class of questions.
In particular, we will allow for time dependence in the modulus of continuity. We will also allow for discontinuous "moduli of continuity",
which will turn out quite useful for the questions of regularization of weak solutions. We postpone the statement of the general criteria
to the Section~\ref{gencrit} since it is fairly technical and requires lengthy notation to state.

The main application that we consider here is the eventual regularization of the solutions to supercritcal Burgers, SQG and modified SQG equation.
For the supercritical ($\alpha<1/2$) Burgers equation, it is known that the solutions can form shocks \cite{KNS}.
It is not known whether solutions of the supercritical SQG and modified SQG equations can form singularities in finite time.
However, for the SQG equation, Silvestre proved \cite{Sil1} that for some $\alpha <1/2$ but sufficiently close to it, the weak solution becomes regular after a finite time
(a similar result for the Burgers equation was proved in \cite{Sil2}). The method of Silvestre is based on Caffarelli-Vasseur technique \cite{CV}.
Recently, Dabkowski \cite{Dabk} proved eventual regularization for arbitrary supercritical SQG ($1/2>\alpha>0$). The method of Dabkowski builds upon the
approach of \cite{KN1}, but employs a more flexible class of test functions and improves some key estimates. The approach uses incompressibility of $u$ in an
essential way, but is not tied to the specific structure of the SQG equation.

Here we will prove eventual regularization in the full range of
supercritical $\alpha$ for Burgers, SQG and modified SQG equations. The method is different from either \cite{Sil1} or \cite{Dabk}, and is more closely
related to the original approach of \cite{KNV}. We now state the main results that we prove. To simplify regularity issues, it will be convenient for us to
work with dissipative regularization of \eqref{as1}:
\begin{equation}\label{asd1}
\theta_t = u \cdot \nabla \theta + (-\Delta)^{\alpha}\theta + \epsilon \Delta \theta,  \,\,\,\theta(x,0)=\theta_0(x).
\end{equation}
The solutions of \eqref{asd1} will be smooth for $t>0$ for every $\epsilon>0,$ for all equations we consider (given, say, $L^2$ initial data).
All estimates we prove will be proved uniformly in $\epsilon>0.$ Then, in the limit $\epsilon \rightarrow 0$ we can obtain a weak solution
via standard procedure (see e.g. \cite{CC1}). It is not hard to show that this weak solution will inherit the uniform estimates obeyed by approximations.
We first state our results for the Burgers equation. 
\begin{theorem}\label{mainbur}
Let $\theta(x,t)$ be solution of \eqref{asd1} set on $\Tm^1$ with $u=\theta,$
$\theta_0 \in L^\infty,$ $0<\alpha<1/2$ and $\epsilon>0.$ Fix $\beta,$ $1>\beta>1-2\alpha.$
Then there exists a time $T=T(\alpha,\beta,\|\theta_0\|_{L^\infty})$
such that $\|\theta(x,t)\|_{C^\beta}$ is uniformly bounded for all times $t \geq T,$ with bound independent of $\epsilon.$
\end{theorem}
The following result can be derived from Theorem~\ref{mainbur}, given results of \cite{KNS} and techniques of \cite{CW} or \cite{CCW}.
\begin{theorem}\label{secbur}
Assume $0<\alpha<1/2$ and $\theta_0 \in H^{3/2-2\alpha}.$ Let $\theta(x,t)$ be weak solution of the one-dimensional Burgers equation
\eqref{asd1} in the periodic setting obtained in the limit $\epsilon \rightarrow 0$. Then there exist $0<T_1(\alpha,\theta_0)\leq T_2(\alpha,\theta_0)<\infty$ such that
$\theta(x,t)$ is smooth for $0<t <T_1$ and $t>T_2.$
\end{theorem}
We will prove Theorems~\ref{mainbur} and \ref{secbur} in Section~\ref{burg}.

The next two Theorems are the analogs of Theorems~\ref{mainbur} and \ref{secbur} for the SQG and modified SQG equations.
\begin{theorem}\label{mainsqg}
Let $\theta(x,t)$ be solution of \eqref{asd1} set on $\Tm^2,$ with $u= \nabla^\perp (-\Delta)^{-\gamma} \theta,$
$\theta_0 \in L^\infty,$ $1/2 \leq \gamma<1,$ $\gamma+\alpha<1,$ $\alpha>0,$ and $\epsilon>0.$ Fix $\beta,$ $2-2\gamma>\beta>2-2\gamma-2\alpha.$
Then there exists a time $T=T(\alpha,\gamma,\beta,\|\theta_0\|_{L^\infty})$
such that $\|\theta(x,t)\|_{C^\beta}$ is uniformly bounded for all times $t \geq T,$ with bound independent of $\epsilon.$
\end{theorem}
\begin{theorem}\label{secsqg}
Assume $1/2 \leq \gamma <1,$ $\alpha+\gamma<1,$ $\alpha>0,$ and $\theta_0 \in H^{1+2\gamma-2\alpha}.$ Let $\theta(x,t)$ be weak solution of the (modified) SQG equation
\eqref{asd1} (set on $\Tm^2$ with $u= \nabla^\perp (-\Delta)^{-\gamma} \theta$) obtained in the limit $\epsilon \rightarrow 0$. Then there exist $0<T_1(\alpha,\gamma,\theta_0)\leq T_2(\alpha,\gamma,\theta_0)<\infty$ such that
$\theta(x,t)$ is smooth for $0<t <T_1$ and $t>T_2.$
\end{theorem}
 In Theorems~\ref{mainsqg} and \ref{secsqg}, the case $\gamma=1/2$ corresponds to the SQG equation, while $\gamma >1/2$ to the modified SQG.
 The SQG case is more complicated for our method than modified SQG case, for the reasons that will be explained in more detail in Section~\ref{sqg}.
 Similarly to the Burgers case, Theorem~\ref{secsqg} follows from Theorem~\ref{mainsqg} given ideas of \cite{CCW,CW,KNS}. This will be explained in more
detail in Section~\ref{modsqg}.

We would like to stress that, similarly to \cite{Sil1,Sil2,Dabk}, the regularization mechanism (or, at least, the mechanism of the proof) is different
from typical long time regularity results due to certain norm of solution becoming small. Rather, the picture is that of regularization cascade spreading from
large to small scales. The proof proceeds by construction of a family of stationary moduli of continuity that are sufficiently strong to imply regularity and are preserved
by the evolution. However, not every initial data obeys a modulus of continuity from this family. Next, we construct related moduli of continuity that have a jump
at zero. For every initial data, we can find a modulus from this family that is obeyed. Finally, using the general criterion proved in the following section,
we prove that this discontinuous modulus of continuity improves to the regular one in a finite time: large scale regularity gradually propagates to smaller scales.
The proof is quite similar for the Burgers and modified SQG equations,
and is trickier in the SQG case. Nevertheless, the families of the moduli that we work with are the same for all equations. The heart of the proof is quite simple
(except harder SQG case) and is contained in Proposition~\ref{statmod} and Lemma~\ref{finlembur} of Section~\ref{burg}.

\section{The General Criterion}\label{gencrit}

In this section we state and prove a general criterion yielding nonlocal maximum principles for active scalars.
We will focus on the case where initial data is periodic.
Let us start by introducing some notation and terminology.
\begin{definition}\label{modcon}
We call a function $\omega(\xi): (0,\infty) \mapsto (0,\infty)$ a modulus of continuity if $\omega$ is increasing,
continuous on $(0,\infty),$ concave, and piecewise $C^2$ with one sided derivatives (possibly infinite at $\xi=0$)
defined at every point in $[0,\infty).$
We say that a function $f(x)$ obeys $\omega$ if $|f(x)-f(y)| < \omega(|x-y|)$
for all $x \ne y.$
\end{definition}
 Observe that, in contrast to \cite{KNV}, we do not define $\omega$ at zero and do not require that $\omega(0)=0.$
Thus some of our moduli of continuity may be obeyed by irregular or even discontinuous functions. 

Next, suppose that we consider the active scalar equation \eqref{as1}, with $u$ determined by $\theta$ in a certain way.
Assume that if $\theta(x)$ obeys some modulus of continuity $\omega,$ it can be proved that
\begin{equation}\label{ucon12}
| (u(x+\xi e)-u(x))\cdot e | \leq \Omega(\xi)
\end{equation}
for every $x \in \Tm^d,$ $\xi \in (0,\infty)$ and every unit vector $e \in \Rm^d.$
If we are dealing with the time dependent modulus of continuity $\omega(\xi,t)$ for $\theta(x,t)$, we will denote the corresponding bound in \eqref{ucon12}
by $\Omega(\xi,t).$ The form of this $\Omega$ will depend on the precise link between $\theta$ and $u.$  Also, define
\begin{equation}\label{dtermest}
D_{\alpha}(\xi) = c_\alpha \left( \int\limits_0^{\xi/2}
\frac{\omega(\xi+2\eta)+\omega(\xi-2\eta)-2\omega(\xi)}{\eta^{1+2\alpha}}\,d\eta
+\int\limits_{\xi/2}^\infty
\frac{\omega(\xi+2\eta)-\omega(2\eta-\xi)-2\omega(\xi)}{\eta^{1+2\alpha}}\,d\eta
\right),
\end{equation}
where $c_\alpha$ are certain positive constants depending only on $\alpha$ to be
described later. If our modulus of continuity $\omega(\xi,t)$ is time dependent, we will reflect this in notation $D_\alpha(\xi,t).$
Observe that $D_\alpha(\xi)$ is always less than or equal to zero 
due to concavity of $\omega.$

The following theorem is the main result of this section.

\begin{theorem}\label{mainthm}
Let $\theta(x,t)$ be a periodic smooth solution of \eqref{as1} with $\epsilon \geq 0.$ Suppose that
$\omega(\xi,t)$ is piecewise $C^1$ in time variable (with one-sided derivatives defined at all points) for each fixed $\xi >0,$ and that for each
fixed $t \geq 0,$ $\omega(\xi,t)$ is a modulus of continuity.
Assume in addition that for each $t \geq 0,$
either $\omega(0+,t)>0,$ or $\partial_\xi \omega(0+,t)=\infty,$ or $\partial^2_{\xi\xi}\omega(0+,t)=-\infty,$
and that $\omega(0+,t),$ $\partial_\xi \omega(0+,t)$ are continuous
in $t$ with values in $\Rm \cup \infty.$
Let initial data $\theta_0(x)$ obey $\omega(\xi,0) \equiv
\omega_0(\xi).$ Then $\theta(x,T)$ obeys modulus of continuity $\omega(\xi,T)$ provided that
$\omega(\xi,t)$ satisfies
\begin{equation}\label{keyineq}
\partial_t \omega(\xi,t) > \Omega(\xi,t) \partial_\xi
\omega(\xi,t) + D_{\alpha}(\xi,t)+2\epsilon \partial^2_{\xi\xi}\omega(\xi,t)
\end{equation}
for all $\xi>0,$ $T \geq t >0$ such that $\omega(\xi,t) \leq 2\|\theta(x,t)\|_{L^\infty}.$
In \eqref{keyineq}, at the points where $\partial_\xi
\omega(\xi,t)$ ($\partial_t
\omega(\xi,t)$) does not exist, the larger (smaller) value of the one-sided derivative should be taken.
\end{theorem}
\noindent \it Remarks. \rm
1. In applications, it is often convenient to take $\omega$
with a jump in the first and higher order derivatives - due to the different nature of balance between dissipation
and nonlinearity for small and large values of $\xi.$
This jump can usually be avoided, but it handily simplifies
the estimates. Thus it is useful to have Theorem~\ref{mainthm} stated in the form allowing less regularity for $\omega.$\\
2. The condition that \eqref{keyineq} holds only for $\xi,t$ for which
$\omega(\xi,t) \leq 2\|\theta(x,t)\|_{L^\infty}$ is natural, since other values of $\xi,t$ are not relevant
for the dynamics. It is useful in some applications such as for example proof of the existence of solutions of the
 critical Burgers equation with rough initial data (see \cite{KNS}). \\ 
3. The proof of Theorem~\ref{mainthm} parallels closely the original argument of \cite{KNV}. The main difference is that
$\Omega(\xi)$ in \cite{KNV} is just the modulus of continuity of $u$ provided that $\theta$ obeys $\omega.$ The improvement provided
by \eqref{ucon12} will be essential for the application to the SQG equation in Section~\ref{sqg}. \\

Thus the regularity properties of an active scalar are related to
supersolutions of a strongly nonlinear Burgers-type equation
\eqref{keyineq}, with key terms determined by the nature of vector
field $u$ and strength of dissipation. Dissipation terms which are
more general than $(-\Delta)^\alpha$ can also be studied.

The first step in the proof of Theorem~\ref{mainthm} is the
following lemma, identifying the scenario how a modulus of
continuity may be lost.

\begin{lemma}\label{scenlemma}
Under conditions of Theorem~\ref{mainthm}, suppose that for some
$t>0$ the solution $\theta(x,t)$ no longer obeys $\omega(\xi,t).$ Then
there must exist $t_1 > 0$ and $x \ne y$ such that for all $t < t_1,$
$\theta(x,t)$ obeys $\omega(\xi,t),$ while
\begin{equation}\label{brscen11} \theta(x,t_1) - \theta(y,t_1) = \omega(|x-y|,t_1). \end{equation}
\end{lemma}
\begin{proof}
Define $t_1$ as the supremum of all times $t$ such that $\theta(x,s)$ obeys $\omega(\xi,s)$ for all $s<t.$
Observe that we have $|\theta(x,t_1)-\theta(y,t_1)| \leq \omega(|x-y|,t_1)$ for all $x,y \in \Tm^d.$ Indeed, if for
some $x,y$ we had $\theta(x,t_1)-\theta(y,t_1)>\omega(|x-y|,t_1),$ then by continuity in time
we have the same inequality for all $t<t_1$ sufficiently close to $t_1,$ which is impossible
by definition of $t_1.$ Consider
\[ f(x,y,t_1)=\frac{|\,\theta(x,t_1)-\theta(y,t_1)|}{\omega(|x-y|,t_1)}, \]
defined for all $x,y \in \Tm^d,$ $x \ne y.$ 
We showed $f(x,y,t_1) \leq 1.$ If $f(x,y,t_1)=1$ for some
$x,y,$ the proof is completed. So assume, on the contrary, that $f(x,y,t_1) <1$ for all $x \ne y.$
We will show that in this case $f(x,y,t)<1$ for all $x \ne y \in \Tm^d$ and $t \in [t_1,t_1+h]$ with some small $h>0,$ contradicting the definition of $t_1.$

First, we show that $f(x,y,t)$ remains less than $1$ for close values of $x,y.$
Namely, we claim there exists $h>0,$ $\kappa >0$ and $\rho<1$ such that if $|x-y| < \kappa,$ then for all $t\in [t_1,t_1+h]$ we have $f(x,y,t)< \rho.$
If $\omega(0+,t_1)>0$ or $\partial_\xi\omega(0+,t_1) =\infty,$ by continuity in time this is immediate. If $\partial_\xi\omega(0+,t_1)$
is finite, then we have $\|\nabla \theta (\cdot,t_1)\|_{L^\infty} < \partial_\xi\omega(0+,t_1).$ This follows from
the condition $\partial^2_{\xi\xi}\omega(0+,t_1)=-\infty$. Indeed, if we had $|\nabla \theta(x,t_1)| =\partial_\xi\omega(0+,t_1)$
for some $x,$ then applying mean value theorem we could find $y$ in a small neighborhood of $x$ such that
$|\theta(x,t_1)-\theta(y,t_1)| > \omega(|x-y|,t_1)$ - a contradiction.
By continuity in time, the bound $\|\nabla \theta (\cdot ,t)\|_{L^\infty} < \partial_\xi\omega(0+,t)$ extends
to a small interval $[t_1,t_1+h].$ Set \[ \rho = \frac12 + \frac12 {\rm sup}_{t \in [t_1,t_1+h]} \frac{\|\nabla \theta (\cdot ,t)\|_{L^\infty}}
{\partial_\xi\omega(0+,t)}<1. \]
Using smoothness of $\theta(x,t)$ and compactness of $\Tm^d$ and $[t_1,t_1+h],$ it is straightforward to show that there exists $\kappa >0$ such that
$f(x,y,t) < \rho$ if $|x-y| < \kappa,$ $t \in [t_1,t_1+h].$
This completes the treatment of $x,y$ which are close.

But now we are left with a continuous function $f(x,y,t)$ on a compact set $K \equiv (\Tm^d \times \Tm^d) \setminus \{ |x-y|<\kappa \},$ and
$f(x,y,t_1)<1$ on $K.$ By continuity in time, there exists a small interval of time past $t_1$ where this inequality is preserved.
This gives us a contradiction with the choice of $t_1.$ Hence there must exist $x,y$ where \eqref{brscen11} holds.
\end{proof}

\begin{proposition}\label{termsests}
Let $\theta$ be a smooth periodic solution of \eqref{as1}.
Assume we are in the situation of Lemma~\ref{scenlemma}, namely that $\theta(x,t_1)-\theta(y,t_1)= \omega(|x-y|,t_1)$
and $t_1$ is the minimal time when $\omega(\xi,t)$ is not obeyed. Then
\begin{equation}\label{gencrkey12}  \partial_t \left. (\theta(x,t)-\theta(y,t)) \right|_{t=t_1} \leq \Omega(|x-y|,t_1) \partial_\xi \omega(|x-y|,t_1)+D_\alpha(|x-y|,t_1)+2\epsilon \partial^2_{\xi\xi}
\omega(|x-y|,t_1). \end{equation}
\end{proposition}
Due to concavity of $\omega,$
the last term we get in \eqref{gencrkey12} is never positive.
Therefore, in applications, we can just ignore this term in all estimates, making sure \eqref{gencrkey12} holds without this term. The resulting $\omega(\xi,t)$
will be obeyed by solutions of \eqref{asd1} independently of $\epsilon \geq 0.$
\begin{proof}
Set $\xi = |x-y|.$ We have
\begin{eqnarray}\label{eq12}
\partial_t \left.(\theta(x,t)-\theta(y,t))\right|_{t=t_1} = \\
\nonumber (u \cdot \nabla)\theta(x,t_1) - (u \cdot \nabla)\theta(y,t_1) - (-\Delta)^\alpha \theta(x,t_1) + (-\Delta)^\alpha \theta(y,t_1)+\epsilon
\Delta \theta(x,t_1)-\epsilon \Delta \theta(y,t_1).
\end{eqnarray}
We claim that the combination of the first two terms can be estimated from the above by $\Omega(\xi,t_1)\partial_\xi \omega(\xi,t_1).$ 
To prove this estimate, let us denote by $l$ the unit vector pointing from point $y$ to point $x.$ First, we claim that $\partial_l \theta(x,t)$
is equal to $\partial_\xi\omega(\xi,t_1).$ Indeed, if not, we could shift point $x$ along the direction $l$ to a new position $x'$ and obtain a pair of points $x',$ $y$
where $\theta(x',t_1)-\theta(y,t_1) > \omega(|x'-y|,t_1),$ a contradiction with the choice of $t_1.$ Similarly, $\partial_l \theta(y,t_1) =\partial_\xi\omega(\xi,t_1).$
Also, for every direction $v$ orthogonal to $l,$
$\partial_v \theta(x,t_1)=0.$ If not, we could move $x$ a little in a direction of increase along $v$ and obtain a contradiction, similarly to the
previous argument. In the same fashion, $\partial_v \theta(y,t_1)=0$ for every $v$ orthogonal to $l.$
Therefore,
\[ \left|(u \cdot \nabla)\theta(x,t_1) - (u \cdot \nabla)\theta(y,t_1) \right|= \left|((u(x,t_1)-u(y,t_1))\cdot l)\partial_\xi \omega(\xi,t_1)\right| \leq \Omega(\xi,t_1)
\partial_\xi \omega(\xi,t_1). \]
In the last step we used \eqref{ucon12} and the fact that by definition $x-y = \xi l.$

Consider now the last two terms. First, we claim that for every
$v$ orthogonal to $l$ we have $\partial^2_{vv} \theta(x,t_1)-\partial_{vv}^2 \theta(y,t_1) \leq 0.$ If not, given that $\partial_v \theta(x,t_1)=\partial_v \theta(y,t_1)=0$,
we could shift both $x$ and $y$ along $v,$ keeping distance between them constant, and obtain a contradiction using mean value theorem. Finally, observe that since $\partial_l \theta(x,t)=\partial_l \theta(y,t)=
\partial_\xi\omega(\xi,t_1),$ in order for $\omega$ not to be violated with a strict inequality at some points near $x,y$ we also need $\partial^2_{ll}\theta(x,t_1)
\leq \partial^2_{\xi\xi} \omega(\xi,t_1)$ and $\partial^2_{ll}\theta(y,t_1) \geq -\partial^2_{\xi\xi}\omega(\xi,t_1).$ Combining all estimates for the second order terms, we get
$\Delta\theta(x) -\Delta\theta(y) \leq 2\partial^2_{\xi\xi}\omega(\xi,t_1).$

To estimate the third and fourth terms, let us introduce the
semigroup $e^{-(-\Delta)^\alpha t}$ generated by $-(-\Delta)^\alpha,$ and its kernel $\cP^{\alpha,d}_t(x)$ where
\begin{equation}\label{Phi1}
\cP^{\alpha,d}_t(x) = t^{-d/2\alpha} \cP^{\alpha,d}(t^{-1/2\alpha}x), \,\,\,\cP^{\alpha,d}(x) =
\frac{1}{2\pi}\int_{{\Rm}^d} \exp (i xk-|k|^{2\alpha}) \,dk.
\end{equation}
It is evident that $\cP^{\alpha,d}(x)$ is radially symmetric and $\int \cP^{\alpha,d}(x)\,dx=1.$  We
will need the following further properties of the function $\cP^{\alpha,d}:$
\begin{equation}\label{Phi}
\cP^{\alpha,d}(x) >0;\,\,\,\partial_r\cP^{\alpha,d}(x) \leq 0,\,\,\,\frac{C_1(\alpha)}{1+|x|^{d+2\alpha}} \leq \cP^{\alpha,d}(x) \leq
\frac{C_2(\alpha)}{1+|x|^{d+2\alpha}}.
\end{equation}
First two properties are not difficult to prove; see e.g.
\cite{Feller} for some results, in particular positivity (Theorem
XIII.6.1). See \cite{BSS} for the third property.

We remark that Caffarelli and Silvestre \cite{CS} provide an alternative way to generate fractional Lapalcian.
This approach has an advantage of having completely explicit kernel, and can also be used here.

The combination of the third and fourth terms in \eqref{eq12} can be written as
\[ - (-\Delta)^\alpha \theta(x,t_1) + (-\Delta)^\alpha \theta(y,t_1)= \lim_{h \rightarrow 0}\frac{1}{h}\left(
(\cP^{\alpha,d}_h * \theta)(x,t_1)-(\cP^{\alpha,d}_h * \theta)(y,t_1) -\theta(x,t_1)+\theta(y,t_1) \right). \]
Here in convolution we use the $\Rm^d$ kernel $\cP^{\alpha,d}_h$ and integrate over $\Rm^d,$ extending $\theta$ periodically -
this is the same as working on $\Tm^d$.
The last two terms on the right hand side are exactly $-\omega(\xi,t_1).$ The first two terms can be estimated from above,
identically to the argument in Section 5 of \cite{KNV}, by
\[ \int_0^{\xi/2} \cP^{\alpha,1}_h[\omega(\xi+2\eta,t_1)+\omega(\xi-2\eta,t_1)]\,d\eta + \int_{\xi/2}^\infty \cP^{\alpha,1}_h(\eta)[\omega(2\eta+\xi,t_1)
-\omega(2\eta-\xi,t_1)]\,d\eta \]
(prior to division by $h$). We will review this derivation in Section~\ref{sqg} where we will need an improvement.
The first two properties \eqref{Phi} are used in this derivation. Finally, subtracting $\omega(\xi,t),$ dividing by $h,$ passing to the limit
and using the third property \eqref{Phi} and scaling \eqref{Phi1} we arrive at $D_\alpha(\xi,t)$ with $c_\alpha$ in \eqref{dtermest} equal to
$C_1(\alpha).$
\end{proof}

Now we are ready to prove Theorem~\ref{mainthm}.
\begin{proof}[Proof of Theorem~\ref{mainthm}]
Assume that Theorem is not true. Then by Lemma~\ref{scenlemma}, there exists $t_1>0$ such that $\theta(x,t)$ obeys $\omega(\xi,t)$ for $t<t_1$
but $\theta(x,t_1)-\theta(y,t_1)= \omega(|x-y|,t_1)$ for some $x,y \in \Tm^d.$ Set $\xi = |x-y|.$ Then we have
\begin{equation}\label{critpoint12} \partial_t \left. \left[ \frac{\theta(x,t)-\theta(y,t)}{\omega(\xi,t)} \right]\right|_{t=t_1} \leq
\frac{\Omega(\xi,t_1) \partial_\xi \omega(\xi,t_1)+D_\alpha(\xi,t_1)+2\epsilon \partial^2_{\xi\xi}\omega(\xi,t_1)-\partial_t \omega(\xi,t_1)}{\omega(\xi,t_1)}. \end{equation}
We used Lemma~\ref{termsests} and the fact that $\theta(x,t_1)-\theta(y,t_1)=\omega(\xi,t_1).$
By \eqref{keyineq}, the numerator of the right hand side of \eqref{critpoint12} is negative. This means that for some time smaller than $t_1,$
$\theta(x,t)-\theta(y,t)$ already exceeded $\omega(\xi,t).$ Contradiction.
\end{proof}



\section{Eventual Regularization in the Supercritical Case: the Burgers equation}\label{burg}

Here we consider the Burgers equation
\begin{equation}\label{fracbur}
\theta_t = \theta \theta_x -(-\Delta)^\alpha \theta + \epsilon \Delta \theta, \,\,\,\theta(x,0)=\theta_0(x).
\end{equation}
The Burgers equation with $\alpha<1/2$ is known to form shocks in finite time (\cite{KNS}). In this section we prove Theorems~\ref{mainbur} and \ref{secbur}, showing
that nevertheless the viscosity solutions of the supercritical Burgers equation with arbitrary $\alpha>0$ become regular after
some fixed time depending on $\alpha$ and the size of initial data. 

We start with a few comments on how Theorem~\ref{mainbur} implies Theorem~\ref{secbur}.
Local existence of solution, smooth for $t>0,$ for the initial data in $H^{3/2-2\alpha}$ has been established in \cite{KNS}.
Furthermore, for the supercritical
SQG equation it has been proved by Constantin and Wu \cite{CW} (in the whole space case) that if we have control of $C^\beta$ norm of the solution uniformly in time with $\beta>1-2\alpha,$
then the solution is regular and can be extended globally. Since the scaling of Burgers equation is similar to that of the SQG, one can expect a similar
result to be true in the case of the Burgers equation. We will state this property in a way fitting our regularization procedure.

\begin{proposition}\label{burreg56}
Suppose $0<\alpha <1/2.$
Assume that a solution $\theta(x,t)$ of \eqref{fracbur} satisfies for all $t \geq T$ the bound $\|\theta(\cdot,t)\|_{C^\beta} \leq B,$
for $\beta > 1-2\alpha.$ Then for every $s>0$ and $t > T$ we also have $\|\theta(\cdot,t)\|_{H^s} \leq f(B,s,t) < \infty.$
The function $f(B,s,t)$ does not depend on the value of $\epsilon$ in \eqref{fracbur}.
\end{proposition}

The proof of Proposition~\ref{burreg56} can be carried out by methods adapted in a straightforward way from \cite{CW} (alternatively, the approach of
\cite{CCW} can also be used here).
Together, Theorem~\ref{mainbur}, Proposition~\ref{burreg56} and standard arguments on approximations imply regularization of solution after some time $T,$
proving Theorem~\ref{secbur}. 

We start the proof of Theorem~\ref{mainbur} by finding some stationary moduli of continuity that are preserved by the evolution in the supercritical Burgers.

\begin{proposition}\label{statmod}
Let $0<\alpha<1/2$ and $1>\beta > 1-2\alpha.$ Define
\begin{equation}\label{modcon1}
\omega(\xi) = \left\{ \begin{array}{ll} H (\xi/\delta)^\beta, & 0 < \xi \leq \delta \\
H, & \xi > \delta \end{array} \right.
\end{equation}
There exists a constant $C_1=C_1(\alpha,\beta)$ such that if $H \leq C_1 \delta^{1-2\alpha},$ the following is true.
If the initial data $\theta_0(x)$ obeys $\omega(\xi),$ the solution of \eqref{fracbur} $\theta(x,t)$ obeys $\omega(\xi)$ for every $t,$
independently of $\epsilon.$
\end{proposition}
The proposition gives us a family of moduli of continuity parameterized by $H$ and $\delta$ that are preserved by the evolution.
Of course, not every initial data will obey some moduli from this family, due to restriction $H \leq C_1 \delta^{1-2\alpha}.$
Nevertheless, what we will show is that eventually solutions corresponding to reasonable initial data start to obey one of the moduli of continuity
given by \eqref{modcon1}, and thus, due to Proposition~\ref{burreg56}, become smooth.
\begin{proof}
We need to prove that the right hand side of \eqref{keyineq} does not exceed zero for $\omega$ given by \eqref{modcon1}, under assumption that $H \leq C_1 \delta^{1-2\alpha}.$
If $\xi > \delta,$ the result is immediate since $\omega'(\xi)=0.$ If $0<\xi \leq \delta,$ the nonlinear term is equal to
\[ \omega(\xi)\omega'(\xi) = \beta H^2 \delta^{-2\beta} \xi^{2\beta-1}. \]
For the dissipation term, we will use just the first part
\begin{equation}\label{dt12c}  c_\alpha \int\limits_0^{\xi/2}
\frac{\omega(\xi+2\eta)+\omega(\xi-2\eta)-2\omega(\xi)}{\eta^{1+2\alpha}}\,d\eta. \end{equation}
Observe that
\[ \omega(\xi+2\eta)+\omega(\xi-2\eta)-2\omega(\xi) = \int_{-2\eta}^{2\eta} (2\eta - |z|) \omega''(\xi+z)\,dz \leq C\omega''(\xi)\eta^2, \]
since $\omega''(\xi)$ is negative and monotone increasing.
Then the expression in \eqref{dt12c} does not exceed
\[ C(\alpha) \omega''(\xi) \xi^{2-2\alpha} = -C(\alpha,\beta)H \delta^{-\beta} \xi^{\beta-2\alpha}.\]
Therefore to have preservation of $\omega,$ we need
\[ \beta H^2 \delta^{-2\beta} \xi^{2\beta-1} < C(\alpha,\beta)H \delta^{-\beta} \xi^{\beta-2\alpha}, \]
which reduces to
\begin{equation}\label{finbal11} H \delta^{-\beta} < C(\alpha,\beta)\xi^{1-\beta-2\alpha}. \end{equation}
Since by assumption $\beta+2\alpha>1,$ the inequality \eqref{finbal11} holds if $H \leq C_1 \delta^{1-2\alpha}$
with appropriately chosen $C_1$ which depends only on $\alpha$ and $\beta.$
\end{proof}

Given $\omega$ defined by \eqref{modcon1}, consider the following derivative family
\begin{equation}\label{omxi0}
\omega(\xi,\xi_0) = \left\{ \begin{array}{ll} \beta H \delta^{-\beta} \xi_0^{\beta-1}\xi + (1-\beta) H \delta^{-\beta} \xi_0^{\beta}, & 0<\xi < \xi_0 \\
H (\xi/\delta)^\beta, & \xi_0 \leq \xi \leq \delta \\
H, & \xi>\delta \end{array} \right.
\end{equation}
where $0 \leq \xi_0<\delta.$ Observe that $\omega(\xi,0)$ coincides with $\omega(\xi)$ given by \eqref{modcon1}. The modulus of continuity $
\omega(\xi,\xi_0)$ is obtained by taking a tangent line to
$\omega(\xi)$ at $\xi = \xi_0$ and replacing $\omega(\xi)$
with this tangent line for $0<\xi<\xi_0.$ Clearly, $\omega(\xi,\xi_0)$ does not tend to zero as $\xi \rightarrow 0.$ Also, it is clear that any bounded initial
data $\theta_0$ obeys $\omega(\xi, \delta)$ provided that $2\|\theta_0\|_{L^\infty} \leq \omega(0,\delta)=(1-\beta)H.$
Thus to prove Theorem~\ref{mainbur}, it suffices to prove the following lemma.
\begin{lemma}\label{finlembur}
Assume that the initial data $\theta_0(x)$ for \eqref{fracbur} obeys $\omega(\xi, \delta).$
Then there exist positive constants $C_{1,2}=C_{1,2}(\alpha,\beta)$ such that if $\xi_0(t)$ is a solution of
\begin{equation}\label{xi0eq}
\frac{d\xi_0}{dt} = -C_2\xi_0^{1-2\alpha}, \,\,\,\xi_0(0)=\delta,
\end{equation}
and $H \leq C_1 \delta^{1-2\alpha},$ then the solution $\theta(x,t)$ obeys $\omega(\xi,\xi_0(t))$ for all $t$ for which $\xi_0(t) \geq 0.$
\end{lemma}
 Theorem~\ref{mainbur} follows from Lemma~\ref{finlembur}, since the solution $\xi_0(t)$ of equation \eqref{xi0eq} reaches zero in finite time $T$.
Therefore, for all $t>T,$ the solution $\theta(x,t)$ obeys $\omega(\xi),$ and thus its $C^\beta$ norm is uniformly bounded.
\begin{proof}
We will check that $\omega(\xi, \xi_0(t))$ satisfies \eqref{keyineq}. We have several ranges to consider. The case $\xi > \delta$ is immediate.
Consider $0<\xi\leq \xi_0(t).$ Then the time derivative is equal to
\begin{equation}\label{timeder}
\partial_t \omega(\xi,\xi_0(t)) = \partial_{\xi_0} \omega(\xi,\xi_0(t))\xi_0'(t) = \beta(1-\beta)\frac{H}{\delta^\beta}(\xi_0^{\beta-1}-\xi_0^{\beta-2}\xi)\xi_0'.
\end{equation}
The nonlinearity is equal to
\begin{equation}\label{nonlin11}
\omega(\xi,\xi_0(t))\partial_\xi \omega(\xi,\xi_0(t)) = \beta\frac{H^2}{\delta^{2\beta}}\xi_0^{2\beta-2}(\beta \xi +(1-\beta)\xi_0).
\end{equation}
Finally, in the dissipation term $D_\alpha(\xi,t)$ we will use just one summand:
\begin{equation}\label{dissip56}
D_{\alpha}(\xi,t) \leq c_\alpha \int\limits_{\xi/2}^\infty
\frac{\omega(\xi+2\eta)-\omega(2\eta-\xi)-2\omega(\xi)}{\eta^{1+2\alpha}}\,d\eta \leq
-C(\alpha) \frac{\omega(0)}{\xi^{2\alpha}}.
\end{equation}
In the last step we used the fact that
\[ \omega(\xi+2\eta)-\omega(2\eta-\xi)-2\omega(\xi) \leq -2\omega(0) \]
due to concavity of $\omega(\xi)-\omega(0).$
First, we claim that one can choose $C_1$ so that
\begin{equation}\label{eq123} \omega(\xi,\xi_0(t))\partial_\xi \omega(\xi,\xi_0(t)) +\frac12 D_{\alpha}(\xi,t) <0.\end{equation}
Indeed from \eqref{nonlin11} and \eqref{dissip56} we find that the sum in \eqref{eq123} does not exceed
\[\beta\frac{H^2}{\delta^{2\beta}}\xi_0^{2\beta-2}(\beta \xi +(1-\beta)\xi_0) - \frac12 c_\alpha \frac{\omega(0)}{\xi^{2\alpha}} \leq \\
\beta \frac{H \xi_0^\beta}{\delta^\beta} \left( \frac{H}{\delta^{1-2\alpha}}\left(\frac{\xi_0}{\delta}\right)^{\beta-1+2\alpha}\left(\frac{\xi}{\xi_0}\right)^{2\alpha}
-C(\alpha,\beta)\right),\]
where $C(\alpha,\beta)$ is an explicit positive constant. Clearly, if $C_1$ is sufficiently small, the expression in the brackets is negative.

Next, we compare the time derivative and dissipative term. We will show that the constant $C_2$ in \eqref{xi0eq} can be chosen sufficiently small so that
\begin{equation}\label{eq124} -\partial_t \omega(\xi,\xi_0(t)) +\frac12 D_{\alpha}(\xi,t) <0;\end{equation}
this will complete consideration of the $0<\xi \leq \xi_0(t)$ range.
Due to \eqref{timeder}, \eqref{dissip56} and \eqref{omxi0}, we find that the expression in \eqref{eq124} does not exceed
\[ -\beta(1-\beta)\frac{H}{\delta^\beta}\xi_0^{\beta-1}\xi_0' - \frac{C(\alpha)(1-\beta)}{2} \frac{H\xi_0^\beta}{\delta^\beta \xi^{2\alpha}} = \\
(1-\beta)\frac{H}{\delta^\beta}\xi_0^{\beta-2\alpha}\left( C_2 \beta - \frac{C(\alpha)}{2} \left(\frac{\xi_0}{\xi}\right)^{2\alpha}\right). \]
Clearly, the last expression is negative if $C_2$ is sufficiently small.

Finally, let us consider the range $\xi_0(t)<\xi \leq \delta.$ Here $\partial_t \omega(\xi, \xi_0(t))=0,$ and it suffices to show that
\begin{equation}\label{eqn111}
\omega(\xi,\xi_0(t))\partial_\xi \omega(\xi,\xi_0(t)) + D_\alpha(\xi,t) <0.
\end{equation}
The nonlinear term $\omega \partial_\xi \omega$ is equal to $\beta H^2 \delta^{-2\beta} \xi^{2\beta-1}.$ For the dissipative term, we find
\begin{eqnarray*}
D_\alpha(\xi,t) \leq c_\alpha \int_0^{\xi/2} \frac{\omega(\xi+2\eta)+\omega(\xi-2\eta)-2\omega(\xi)}{\eta^{1+2\alpha}}\,d\eta \leq c_\alpha
\int_{-\xi}^\xi \omega''(\xi+s) \int_{|s|/2}^{\xi/2} \frac{2\eta -|s|}{\eta^{1+2\alpha}}\,d\eta ds \leq \\
\frac13 \int_{-\xi/2}^{\xi/2} \omega''(\xi+s) \int_{\xi/3}^{\xi/2} \frac{d\eta}{\eta^{2\alpha}}\,ds \leq -C(\alpha,\beta) \xi^{\beta-2\alpha}\frac{H}{\delta^\beta}.
\end{eqnarray*}
Combining these estimates, we see that the difference in \eqref{eqn111} does not exceed
\[ \beta\frac{H}{\delta^\beta}\xi^{\beta-2\alpha} \left( \frac{H}{\delta^{1-2\alpha}}\left(\frac{\xi}{\delta}\right)^{\beta-1+2\alpha} - C(\alpha,\beta)\right). \]
Since $\beta+2\alpha>1,$ the last expression is negative provided that the constant $C_1$ in Lemma~\ref{finlembur} was chosen sufficiently small.
\end{proof}

\section{Eventual Regularization in the Supercritical Case: the Modified SQG equation}\label{modsqg}

In this section, we will consider the modified SQG equation
\begin{equation}\label{modsqg1}
\theta_t = (u \cdot \nabla)\theta -(-\Delta)^\alpha \theta + \epsilon \Delta \theta,\,\,\,\theta(x,0)=\theta_0(x),
\end{equation}
where $u = \nabla^\perp (-\Delta)^{-\gamma} \theta,$ $x \in \Tm^2,$ $1/2 < \gamma <1,$ $\epsilon>0.$
We postpone the more difficult SQG case ($\gamma=1/2$) to the next section. The supercritical range of interest to
us corresponds to $0 < \alpha + \gamma < 1.$ Similarly to the Burgers case, uniform in time control of a sufficiently strong
$C^\beta$ norm gives global regularity of solution.
\begin{proposition}\label{conregmsqg} Suppose $\alpha +\gamma <1,$ $1/2 \leq \gamma < 1,$ $\alpha>0.$
Assume that a solution $\theta(x,t)$ of \eqref{modsqg1} satisfies for all $t \geq T$ the bound $\|\theta(\cdot,t)\|_{C^\beta} \leq B,$
for $\beta > 2-2\gamma-2\alpha.$ Then for every $s>0$ and $t>T$ we also have $\|\theta(\cdot,t)\|_{H^s} \leq f(B,s,t)<\infty.$
The function $f(B,s,t)$ does not depend on the value of $\epsilon$ in \eqref{modsqg1}.
\end{proposition}

This result can be proved in a straightforward way using ideas of \cite{CW} or \cite{CCW}. The existence of local solution, smooth for $t>0,$
for initial data in $H^{1+2\gamma-2\alpha}$ can be proven similarly to results of \cite{KNS} for the Burgers equation.
Hence, Theorem~\ref{secsqg} will follow from Theorem~\ref{mainsqg}.

The proof of Theorem~\ref{mainsqg} is quite similar to the proof of Theorem~\ref{mainbur} given a couple of auxiliary statements.
First, we need a lemma that provides control of $u$ given control of $\theta.$
\begin{lemma}\label{msqgucon}
Assume that $\theta$ obeys modulus of continuity $\omega.$ Then $u  = \nabla^\perp (-\Delta)^{-\gamma} \theta,$ $1/2<\gamma<1,$ obeys modulus of continuity
\begin{equation}\label{modconumsqg}
\Omega(\xi) = A \left( \int_0^\xi \frac{\omega(\eta)}{\eta^{2-2\gamma}}\,d\eta + \xi \int_\xi^\infty \frac{\omega(\eta)}{\eta^{3-2\gamma}}\,d\eta \right)
\end{equation}
\end{lemma}
\begin{proof}
The proof of this lemma is identical to the proof of a similar Lemma in Section 2 of \cite{KNV}. We only need to notice that the Fourier transform
of $|k|^{-2\gamma}$ is equal to $c_\gamma |x|^{-3+2\gamma}.$
\end{proof}

To control solutions of \eqref{modsqg1}, we will use the same family of moduli of continuity \eqref{omxi0} as for the Burgers equation.
Given the next Lemma, the rest of the proof becomes parallel to the Burgers case.
\begin{lemma}\label{Omestmsqg}
Assume that $1/2 < \gamma <1,$ $0<\beta < 2-2\gamma.$ Then $\Omega(\xi,\xi_0)$ given by \eqref{modconumsqg},
corresponding to $\omega(\xi,\xi_0)$ given by \eqref{omxi0}, satisfies
\begin{equation}\label{ourOmb}
 \Omega(\xi,\xi_0) \leq C \xi^{2\gamma-1}\omega(\xi,\xi_0),
\end{equation}
for all $0 < \xi \leq \delta.$ This holds for all choices of parameters $H>0,\delta>0,$ $0 \leq \xi_0 \leq \delta$ in \eqref{omxi0}. The constant $C$ in \eqref{ourOmb}
depends only on $\gamma$ and $\beta.$
\end{lemma}
\begin{proof}
The first summand in \eqref{modconumsqg} is easy to estimate:
\[ \int_0^\xi \frac{\omega(\eta,\xi_0)}{\eta^{2-2\gamma}}\,d\eta \leq \omega(\xi,\xi_0) \int_0^\xi \frac{1}{\eta^{2-2\gamma}}\,d\eta \leq  \frac{1}{2\gamma -1}
\omega(\xi,\xi_0) \xi^{2\gamma-1}. \]
Consider now the second summand and suppose first that $0<\xi \leq \xi_0.$ Splitting integration into two regions, we estimate
\[ \xi \int_\xi^{\xi_0} \frac{\omega(\eta,\xi_0)}{\eta^{3-2\gamma}}\,d\eta \leq \frac{1}{2-2\gamma}\omega(\xi_0,\xi_0)\xi^{2\gamma-1} \leq  \frac{1}{(2-2\gamma)(1-\beta)}\omega(\xi,\xi_0)\xi^{2\gamma-1}. \]
Next,
\begin{eqnarray*} \xi\int_{\xi_0}^\infty  \frac{\omega(\eta,\xi_0)}{\eta^{3-2\gamma}}\,d\eta  \leq  \xi\int_{\xi_0}^\infty  \frac{H}{\delta^\beta \eta^{3-2\gamma-\beta}}\,d\eta
= \frac{H}{(2-2\gamma-\beta)\delta^\beta} \xi \xi_0^{\beta+2\gamma-2} \\ \leq  \frac{1}{2-2\gamma-\beta}\xi^{2\gamma-1} \omega(\xi_0,\xi_0)\leq C(\beta,\gamma)\xi^{2\gamma-1} \omega(\xi,\xi_0).
\end{eqnarray*}
Finally, assume that $\xi_0 < \xi \leq \delta.$ Then
\[ \xi \int_\xi^\infty \frac{\omega(\eta,\xi_0)}{\eta^{3-2\gamma}}\,d\eta = \frac{H \xi}{\delta^\beta} \int_\xi^\infty \frac{1}{\eta^{3-2\gamma-\beta}}\,d\eta =
\frac{1}{2-2\gamma-\beta} \omega(\xi,\xi_0)\xi^{2\gamma-1}. \]
\end{proof}

The following is the analog of the key Lemma~\ref{finlembur}. Theorem~\ref{mainsqg} is an immediate consequence of this Lemma.

\begin{lemma}\label{finlemmsqg}
Suppose that $1/2 \leq \gamma <1,$ $\alpha>0,$ $\alpha + \gamma <1,$ $2-2\gamma-2\alpha  < \beta < 2-2\gamma.$
Assume that the initial data $\theta_0(x)$ for \eqref{modsqg1} obeys $\omega(\xi, \delta).$
There exist positive constants $C_{1,2}=C_{1,2}(\alpha,\beta,\gamma)$ such that if $\xi_0(t)$ is a solution of
\begin{equation}\label{xi0eq1}
\frac{d\xi_0}{dt} = -C_2\xi_0^{1-2\alpha}, \,\,\,\xi_0(0)=\delta,
\end{equation}
and $H \leq C_1 \delta^{2-2\alpha-2\gamma},$ then the solution $\theta(x,t)$ obeys $\omega(\xi,\xi_0(t))$ for all $t$ for which $\xi_0(t) \geq 0.$
\end{lemma}
For the case $1/2 < \gamma <1$ the proof, given Lemmas~\ref{msqgucon} and \ref{finlemmsqg}, is quite similar to the Burgers case.
We leave details for the interested reader. The case $\gamma =1/2$ is different (Lemma~\ref{msqgucon} does not apply since the first integral on the right hand side of
\eqref{modconumsqg} diverges for $\omega(\xi)$ not vanishing at zero).
We address this case in the next section.

\section{Eventual Regularization in the Supercritical Case: the SQG equation}\label{sqg}

The approach we employed to prove eventual regularization for the Burgers and modified SQG equation seems hopeless for the SQG equation.
Indeed, $u$ is a Riesz transform of $\theta$ in the SQG case. If we only have $L^\infty$ bounds for $\theta,$ then a-priori we only know that $u$
is in BMO, and hence can have logarithmic singularities. But then we do not have $L^\infty$ control over $u$ at any scale, which spells trouble for trying
to estimate the nonlinear term: there is no satisfactory $\Omega(\xi)$ for $\omega$ with a jump at zero.
While this is all true, there is a reserve that we can try to use: a better estimate of the dissipation term.

In this section, we will prove Lemma~\ref{finlemmsqg} and Theorem~\ref{mainsqg} for the $\gamma=1/2$ case.
To prove Lemma~\ref{finlemmsqg} for $\gamma=1/2$ case, we will need as before to rule out the breakthrough scenario. Consider the minimal time $t_1$ for which
there exist $x,y \in \Tm^2$ such that $\theta(x,t_1)-\theta(y,t_1)=\omega(\xi,\xi_0(t_1)),$ where $\xi=|x-y|.$
First, we need an improved estimate on the contribution of dissipation to $\partial_t(\theta(x,t)-\theta(y,t)).$
The idea of the argument is as follows.
It will be clear from the estimates in the proof of Lemma~\ref{dissimpest} below that the expression \eqref{dtermest} for $D_\alpha(\xi,t)$ comes from
diffusion "parallel" to the direction $x-y.$ This is the minimal contribution of the diffusion part given that we are in a breakthrough scenario. But this contribution
is realized when $\theta$ does not depend on direction orthogonal to $x-y$ ($\theta$ is equal to $\frac12 \omega(2\eta),$ where $\eta$ is the coordinate
in $x-y$ direction and $\omega$ is taken to be odd for negative values). Hence in this weakest dissipation scenario, the fluid velocity in the direction $x-y$ vanishes,
and nonlinear term does not pose any danger. To generate some fluid motion in direction $x-y,$ we need $\theta$ to vary in the direction orthogonal
to $x-y.$ This will produce additional dissipation, and we can try to use this extra dissipation to help control the nonlinear term. 
This turns out to be possible due to the structure of nonlinearity.
\begin{lemma}\label{dissimpest}
Assume that $x,$ $y,$ $\xi,$ and $t_1$ are as described above. Then
\begin{equation}\label{dissest56}
 -(-\Delta)^\alpha \theta(x,t_1)+(-\Delta)^\alpha \theta(y,t_1) \leq D_\alpha(\xi,t_1)+D^\perp_\alpha(\xi,t_1),
\end{equation}
where
\begin{equation}\label{dissperp}
D^\perp_\alpha(\xi,t_1) \leq -C \int_{(\frac12-c)\xi}^{(\frac12+c)\xi}d\eta \int_{0}^{c\xi} \frac{2\omega(2\eta,\xi_0(t_1))-\theta(\eta,\nu,t_1)+\theta(-\eta,\nu,t_1)-
\theta(\eta,-\nu,t_1)+\theta(-\eta,-\nu,t_1)}{\left(
\left( \frac{\xi}{2}-\eta \right)^2 +\nu^2 \right)^{1+\alpha}}\,d\nu.
\end{equation}
Here $C,c>0$ are fixed constants that may depend only on $\alpha.$
\end{lemma}
\begin{proof}
Recall that we denote by ${\cP}^{\alpha,d}_h(x)$ the kernel corresponding to the operator $\exp(-(-\Delta)^\alpha h)$ in dimension $d.$
Then in our case
\[ -(-\Delta)^\alpha \theta(x,t_1)+(-\Delta)^\alpha \theta(y,t_1) = \lim_{h \rightarrow 0} \frac1h \left( \cP^{\alpha,2}_h * \theta(x,t_1) - \theta(x,t_1)
 -\cP^{\alpha,2}_h * \theta(y,t_1) +\theta(y,t_1) \right). \]
Note that by assumption, $\theta(y,t_1)-\theta(x,t_1) = \omega(\xi,\xi_0(t_1)).$ Let us estimate the difference of the two remaining terms, omitting time dependence
for the sake of brevity. For the rest of the argument, $\omega(\xi) \equiv \omega(\xi,\xi_0(t_1)).$ Without loss of generality, we can align one of the coordinate axes
with $x-y,$ and set $x=(\xi/2,0)$ and $y=(-\xi/2,0).$ Then
\begin{eqnarray*} \cP^{\alpha,2}_h * \theta(x) -\cP^{\alpha,2}_h * \theta(y) =
\iint_{\Rm^2}
[\cP_h^{\alpha,2}(\tfrac{\xi}{2}-\eta,\nu)-\cP_h^{\alpha,2}(-\tfrac{\xi}{2}-\eta,\nu)]\theta(\eta,\nu)\,
d\eta d\nu
\\
=\int_{\Rm}d\nu\int_0^\infty
[\cP_h^{\alpha,2}(\tfrac{\xi}{2}-\eta,\nu)-\cP_h^{\alpha,2}(-\tfrac{\xi}{2}-\eta,\nu)]
[\theta(\eta,\nu)-\theta(-\eta,\nu)]\,d\eta
\\
=
\int_{\Rm}d\nu\int_0^\infty
[\cP_h^{\alpha,2}(\tfrac{\xi}{2}-\eta,\nu)-\cP_h^{\alpha,2}(-\tfrac{\xi}{2}-\eta,\nu)]
\omega(2\eta)\,d\eta + \\
\int_{\Rm}d\nu\int_0^\infty
[\cP_h^{\alpha,2}(\tfrac{\xi}{2}-\eta,\nu)-\cP_h^{\alpha,2}(-\tfrac{\xi}{2}-\eta,\nu)]
[\theta(\eta,\nu)-\theta(-\eta,\nu)-\omega(2\eta)]\,d\eta \\
\equiv D_{\alpha,h}^{||} (\xi) + D^\perp_{\alpha,h} (\xi).
\end{eqnarray*}
The term corresponding to $D_{\alpha,h}^{||}(\xi)$ can be further rewritten as
\begin{eqnarray*}
D_{\alpha,h}^{||}(\xi) = \int_{\Rm}d\nu\int_0^\infty
[\cP_h^{\alpha,2}(\tfrac{\xi}{2}-\eta,\nu)-\cP_h^{\alpha,2}(-\tfrac{\xi}{2}-\eta,\nu)]
\omega(2\eta)\,d\eta \\
= \int_0^\infty[\cP_h^{\alpha,1}(\tfrac{\xi}{2}-\eta)-\cP_h^{\alpha,1}(-\tfrac{\xi}{2}-\eta)]
\omega(2\eta)\,d\eta \\
=
\int_0^\xi \cP_h^{\alpha,1}(\tfrac{\xi}{2}-\eta)\omega(2\eta)\,d\eta+
\int_0^\infty \cP_h^{\alpha,1}(\tfrac{\xi}{2}+\eta)[\omega(2\eta+2\xi)-\omega(2\eta)]\,d\eta \\
= \int_0^{\frac{\xi}{2}}\cP_h^{\alpha,1}(\eta)[\omega(\xi+2\eta)+\omega(\xi-2\eta)]\,d\eta+
\int_{\frac{\xi}{2}}^\infty
\cP_h^{\alpha,1}(\eta)[\omega(2\eta+\xi)-\omega(2\eta-\xi)]\,d\eta\,.
\end{eqnarray*}
The properties \eqref{Phi} of the kernel $P_h^{\alpha,1}(\eta)$ are used in this process.
Recalling that $\int_0^\infty \cP_h^{\alpha,1}(\eta)\,d\eta = \frac12,$ subtracting
$\omega(\xi)$ and passing to the limit $h \rightarrow 0,$ we recover expression
\eqref{dtermest} for $D_\alpha(\xi)$ (thus $D_\alpha(\xi) = \lim_{h \rightarrow 0} \frac1h (D_{\alpha,h}^{||}(\xi)-\omega(\xi))$).

Consider now the term
\begin{equation}\label{diffperp11} D^\perp_{\alpha,h} (\xi)= \int_{\Rm}d\nu\int_0^\infty
[\cP_h^{\alpha,2}(\tfrac{\xi}{2}-\eta,\nu)-\cP_h^{\alpha,2}(-\tfrac{\xi}{2}-\eta,\nu)]
[\theta(\eta,\nu)-\theta(-\eta,\nu)-\omega(2\eta)]\,d\eta. \end{equation}
Observe that due to \eqref{Phi} the first factor under the integral is always strictly positive, while the second factor is
by our assumptions always less than or equal to zero. We will need just a small part of the integral \eqref{diffperp11} -
the part near the dangerous point $\eta=\xi/2.$ Recall that due to \eqref{Phi}, we have
\[ \frac{C_1(\alpha)h}{(\eta^2+\nu^2+h^2)^{1+\alpha}} \leq \cP^{\alpha,2}_h(\eta,\nu) \leq \frac{C_2(\alpha)h}{(\eta^2+\nu^2+h^2)^{1+\alpha}}. \]
Choose a constant $c(\alpha),$ $1/4>c(\alpha)>0,$ which is sufficiently small so that
\begin{equation}\label{cdef55} \cP_h^{\alpha,2}(\tfrac{\xi}{2}-\eta,\nu)-\cP_h^{\alpha,2}(-\tfrac{\xi}{2}-\eta,\nu) \geq \frac12 \cP_h^{\alpha,2}(\tfrac{\xi}{2}-\eta,\nu)
\end{equation}
if $|\nu| \leq c\xi$ and $|\eta-\frac{\xi}{2}| \leq c\xi.$
The contribution of the "perpendicular" diffusion does not exceed
\begin{eqnarray}\nonumber
D^\perp_\alpha(\xi) = \lim_{h \rightarrow 0} \frac1h D^\perp_{\alpha,h} (\xi) = \\
\lim_{h \rightarrow 0} \frac1h\int_{\Rm}\,d\nu\int_0^\infty  \nonumber
[\cP_h^{\alpha,2}(\tfrac{\xi}{2}-\eta,\nu)-\cP_h^{\alpha,2}(-\tfrac{\xi}{2}-\eta,\nu)]
[\theta(\eta,\nu)-\theta(-\eta,\nu)-\omega(2\eta)]\,d\eta \\ \nonumber
\leq \lim_{h \rightarrow 0} \frac{1}{2h} \int_{-c\xi}^{c\xi}d\nu \int_{(\frac12-c)\xi}^{(\frac12+c)\xi}
\cP_h^{\alpha,2}(\tfrac{\xi}{2}-\eta,\nu)[\theta(\eta,\nu)-\theta(-\eta,\nu)-\omega(2\eta)]\,d\eta \\
\leq C(\alpha)\int_{-c\xi}^{c\xi}d\nu \int_{(\frac12-c)\xi}^{(\frac12+c)\xi} \nonumber
\frac{\theta(\eta,\nu)-\theta(-\eta,\nu)-\omega(2\eta)}{\left(
\left( \frac{\xi}{2}-\eta \right)^2 +\nu^2 \right)^{1+\alpha}}\,d\eta \leq \\
-C(\alpha)\int_{0}^{c\xi}d\nu \int_{(\frac12-c)\xi}^{(\frac12+c)\xi}
\frac{2\omega(2\eta)-\theta(\eta,\nu)+\theta(-\eta,\nu)-\theta(\eta,-\nu)+\theta(-\eta,-\nu)}{\left(
\left( \frac{\xi}{2}-\eta \right)^2 +\nu^2 \right)^{1+\alpha}}\,d\eta. \label{diffperp09}
\end{eqnarray}
Notice that despite non-integrable singularity in the denominator, all integrals in the calculations converge absolutely
due to cancelation in the numerator ($\theta(x)$ is smooth and $\theta(\xi/2,0)-\theta(-\xi/2,0)-\omega(\xi)=0$).
\end{proof}
Now consider the nonlinear term. Let us denote $l$ the direction along $\eta$ axis (parallel to $x-y$). We will continue
for now to suppress dependence on time in notation to avoid too cumbersome expressions. In particular, we will use
notation $\omega(\xi)$ for $\omega(\xi,\xi_0(t_1)).$
Recall from the proof of Proposition~\ref{termsests} that
\[
 \left|(u \cdot \nabla)\theta(x) - (u \cdot \nabla)\theta(y) \right|= \left|(u(x)-u(y))\cdot l\right| \omega'(\xi).
\]
Let us define
\begin{equation}\label{Om14}
\Omega(\xi) = A \left( -\xi^{2\alpha}D^\perp_\alpha(\xi)+\xi \int_\xi^\infty \frac{\omega(r)}{r^2}\,dr + \omega(\xi) \right),
\end{equation}
where $A$ is a sufficiently large constant that may depend only on $\alpha.$
The following is the key estimate of the nonlinear term.
\begin{lemma}\label{nonlinsqg234}
Suppose that $u = \nabla^\perp (-\Delta)^{-1/2}\theta,$ $\omega$ is a modulus of continuity, and
$x,$ $y$ and $l$ are as above (we are in a breakthrough scenario). Then we have
\begin{equation}\label{ub341}  \left|(u(x)-u(y))\cdot l\right| \leq \Omega(\xi). \end{equation}
\end{lemma}
The key observation here is that $(u(x)-u(y))\cdot l$ cannot be controlled well by $\omega$ if $\omega$ is discontinuous. Of course, since we are working
with regularization ($\epsilon>0$), $u$ is never infinite - but there are no good uniform in $\epsilon$ bounds for $u$ in terms of $\omega.$ However, the part of $u$
not controlled by $\omega$ comes from integration over a small neighborhood close to the kernel singularity, and this part is dominated by the "perpendicular" part
of dissipation.
\begin{proof}
Set $z=(\eta,\nu).$
Let us recall that
\begin{equation}\label{udiff56}
u(x)-u(y) = C \left( P.V.\int \frac{(x-z)^\perp}{|x-z|^3}\,\theta(z)\,dz - P.V.\int \frac{(y-z)^\perp}{|y-z|^3}\,\theta(z)\,dz \right),
\end{equation}
where $C$ is a fixed constant that we will omit in the future. Here the integration is over $\Rm^2$ and $\theta(z)$ is the periodization of the function
defined on the torus.
We will split the estimate into several parts. Let $\tilde{x} = \frac{x+y}{2}.$ Then
\begin{eqnarray*}
 \left| \int_{|x-z| \geq 2\xi} \frac{(x-z)^\perp}{|x-z|^3}\,\theta(z)\,dz - \int_{|y-z| \geq 2\xi} \frac{(y-z)^\perp}{|y-z|^3}\,\theta(z)\,dz \right| = \\
  \left| \int_{|x-z| \geq 2\xi} \frac{(x-z)^\perp}{|x-z|^3}(\theta(z)-\theta(\tilde{x}))\,dz - \int_{|y-z| \geq 2\xi} \frac{(y-z)^\perp}{|y-z|^3}(\theta(z)
  -\theta(\tilde{x}))\,dz \right| \\ \leq \int_{|\tilde{x}-z| \geq 3\xi} \left|\frac{(x-z)^\perp}{|x-z|^3}-\frac{(y-z)^\perp}{|y-z|^3} \right||\theta(z)-\theta(\tilde{x})|\,dz
  \\ + \int_{3\xi/2 \leq |\tilde{x}-z| \leq 3\xi} \left( \frac{(x-z)^\perp}{|x-z|^3}+\frac{(y-z)^\perp}{|y-z|^3} \right)|\theta(z)-\theta(\tilde{x})|\,dz.
\end{eqnarray*}
Now when $|\tilde{x}-z| \geq 3\xi,$ we have
\[  \left|\frac{(x-z)^\perp}{|x-z|^3}-\frac{(y-z)^\perp}{|y-z|^3} \right| \leq \frac{C|x-y|}{|\tilde{x}-z|^3}. \]
Since $\xi=|x-y|,$ the first integral is estimated from above by $C\xi \int_{\xi}^\infty \frac{\omega(r)}{r^2}\,dr.$ The second integral
can be estimated from above by $C\omega(3\xi) \leq 3C \omega(\xi).$

Next, let us denote by $Q_x$ and $Q_y$ the squares with centers at $x$ and $y$ respectively and side length $2c\xi.$ Observe that
\[ \int_{|x-z| \leq 2\xi,\,\,z \notin Q_x } \frac{(x-z)^\perp}{|x-z|^3}\,\theta(z)\,dz =
\int_{|x-z| \leq 2\xi,\,\,z \notin Q_x } \frac{(x-z)^\perp}{|x-z|^3}(\theta(z)-\theta(x))\,dz  \leq
\int_{c\xi}^{2\xi} \frac{\omega(r)}{r}\,dr  \leq C \omega(\xi).\]
A similar estimate holds for the corresponding integral with $y$ instead of $x.$

Finally, let us consider the contribution of the integrals over $Q_x$ and $Q_y.$ Here we will recall that what we really need to estimate is $(u(x)-u(y))\cdot l.$
Of course, taking inner product with $l$ does not spoil any of our previous estimates. Here is what remains to be estimated:
\begin{eqnarray} \nonumber
\left| \int_{Q_x} \frac{(x-z)^\perp \cdot l}{|x-z|^3}\,\theta(z)\,dz - \int_{Q_y} \frac{(y-z)^\perp \cdot l}{|y-z|^3}\,\theta(z)\,dz \right| = \\ \nonumber
\left| \int\limits_{(\frac12-c)\xi}^{(\frac12+c)\xi}d\eta \int\limits_{-c\xi}^{c\xi} \frac{\nu}{\left(\left(\frac{\xi}{2}-\eta\right)^2+\nu^2\right)^{3/2}}\theta(\eta,\nu)\,d\nu
-\int\limits_{(-\frac12-c)\xi}^{(-\frac12+c)\xi}d\eta \int\limits_{-c\xi}^{c\xi} \frac{\nu}{\left(\left(\frac{\xi}{2}+\eta\right)^2+\nu^2\right)^{3/2}}\theta(\eta,\nu)\,d\nu\right| \\ \nonumber
= \left| \,\int\limits_{(\frac12-c)\xi}^{(\frac12+c)\xi}d\eta \int\limits_{-c\xi}^{c\xi} \frac{\nu}{\left(\left(\frac{\xi}{2}-\eta\right)^2+\nu^2\right)^{3/2}}(\theta(\eta,\nu)-
\theta(-\eta,\nu))\,d\nu \right| \\ \label{ubound345} =\left| \,\int\limits_{(\frac12-c)\xi}^{(\frac12+c)\xi}d\eta
\int\limits_{0}^{c\xi} \frac{\nu}{\left(\left(\frac{\xi}{2}-\eta\right)^2+\nu^2\right)^{3/2}}(\theta(\eta,\nu)-
\theta(-\eta,\nu)-\theta(\eta,-\nu)+\theta(-\eta,-\nu))\,d\nu \right|
\end{eqnarray}
We now claim that \eqref{ubound345} does not exceed $-C\xi^{2\alpha}D^\perp_\alpha(\xi).$ Indeed, compare \eqref{ubound345} and \eqref{diffperp09}. First, within the region of
integration in these integrals, we have
\[  0<\frac{\nu}{\left(\left(\frac{\xi}{2}-\eta\right)^2+\nu^2\right)^{3/2}} \leq \frac{C\xi^{2\alpha}}{\left(
\left( \frac{\xi}{2}-\eta \right)^2 +\nu^2 \right)^{1+\alpha}}. \]
Secondly, we have
\[ |\theta(\eta,\nu)-\theta(-\eta,\nu)-\theta(\eta,-\nu)+\theta(-\eta,-\nu)| \leq 2\omega(2\eta)-\theta(\eta,\nu)+\theta(-\eta,\nu)-\theta(\eta,-\nu)+\theta(-\eta,-\nu). \]
The latter inequality can be checked directly using that $2\theta(\eta,\nu)-2\theta(-\eta,\nu) \leq 2\omega(2\eta)$ and $2\theta(-\eta,-\nu)-2\theta(\eta,-\nu) \leq 2\omega(2\eta).$
Together, these two estimates imply the needed bound. Collecting all our estimates together, we get \eqref{ub341}.
\end{proof}
Moreover, we can simplify our expression \eqref{Om14} for $\Omega(\xi).$
\begin{lemma}\label{ll23}
For every modulus of continuity from our family $\omega(\xi,\xi_0(t)),$ we have
\[ \xi \int_{\xi}^\infty \frac{\omega(r,\xi_0)}{r^2}\,dr \leq C(\beta)\omega(\xi,\xi_0), \]
with constant $C(\beta)$ depending only on $\beta.$
\end{lemma}
\begin{proof}
We will omit $\xi_0$ from notation $\omega(\xi,\xi_0)$ in the following estimates.
Consider $\xi \geq \delta.$ Then
\[ \xi \int_\xi^\infty \frac{\omega(r)}{r^2}\,dr = \xi \int_\xi^\infty \frac{H}{r^2}\,dr = H = \omega(\xi). \]
Next, consider $\xi_0 \leq \xi < \delta.$ Then
\[ \xi \int_\xi^\infty \frac{\omega(r)}{r^2}\,dr = \xi \int_\xi^\delta \frac{H}{\delta^\beta r^{2-\beta}}\,dr + \xi \int_\delta^\infty \frac{H}{r^2}\,dr
=   \frac{1}{1-\beta} H \frac{\xi^\beta}{\delta^\beta} + H \frac{\xi}{\delta} \left(1 - \frac{1}{1-\beta} \right) \leq \frac{1}{1-\beta} \omega(\xi).\]
Finally, consider $0<\xi < \xi_0.$ Here
\begin{eqnarray*}
\xi \int_\xi^\infty \frac{\omega(r)}{r^2}\,dr \leq \xi \int_\xi^{\xi_0} \frac{\omega(\xi_0)}{r^2}\,dr + \xi \int_{\xi_0}^\delta \frac{H}{\delta^\beta r^{2-\beta}}\,dr +
\xi \int_\delta^\infty \frac{H}{r^2}\,dr = \\ \xi\omega(\xi_0) \left( \frac{1}{\xi} - \frac{1}{\xi_0} \right) + \frac{H}{1-\beta} \left( \frac{\xi_0^\beta \xi}{\delta^\beta \xi_0}
- \frac{\xi}{\delta} \right) + H \frac{\xi}{\delta} \leq \left(\frac{1}{1-\beta}+1\right)\omega(\xi_0) \leq \frac{2}{(1-\beta)^2} \omega(\xi).
\end{eqnarray*}
\end{proof}
 Therefore, for our family of moduli of continuity, we can replace \eqref{Om14} with
 \begin{equation}\label{Om15}
 \Omega(\xi) = A(\omega(\xi)-\xi^{2\alpha}D^\perp_\alpha(\xi)),
 \end{equation}
and \eqref{ub341} will still hold.

Now we are ready to prove Theorem~\ref{mainsqg}, case $\gamma=1/2.$
\begin{proof}[Proof of Theorem~\ref{mainsqg}, case $\gamma=1/2$]
Since the moduli of continuity we are using are the same as in the Burgers case, there is no difference in the estimate for
$\partial_t \omega(\xi,\xi_0(t)):$ we still have $\partial_t \omega(\xi,\xi_0(t)) > \frac12 D_\alpha(\xi,t)$ at $t=t_1$ provided that the constant
$C_2$ in the statement of Lemma~\ref{finlemmsqg} is taken sufficiently small. Taking into account our
improved estimates on the nonlinear and dissipative terms, Lemmas~\ref{nonlinsqg234} and \ref{dissimpest}, it remains to prove that
\begin{equation}\label{keyin331} \Omega(\xi,t_1)\partial_\xi \omega(\xi,\xi_0(t_1)) \leq -\frac{1}{2}D_\alpha(\xi,t_1)-D^\perp_\alpha(\xi,t_1) \end{equation}
for all positive $\xi,$ provided that the constant $C_1$ in the statement of Lemma~\ref{finlemmsqg} is taken sufficiently small.
Due to \eqref{Om15}, we can take $\Omega(\xi,t_1) = A(\omega(\xi,\xi_0(t_1))-\xi^{2\alpha}D^\perp_\alpha(\xi,t_1)).$
In a way completely parallel to the proof of Theorem~\ref{mainbur}, we can show that for all $0 < \xi < \delta,$
\[  A\omega(\xi,\xi_0(t_1))\partial_\xi \omega(\xi,\xi_0(t_1)) \leq -\frac12 D_\alpha(\xi,t_1) \]
if $C_1$ is sufficiently small. 
Finally, we also have that for all $0<\xi<\delta,$
\[
-A\xi^{2\alpha} D^\perp_\alpha(\xi,t_1) \partial_\xi \omega(\xi,\xi_0(t_1)) \leq -A \beta \frac{H}{\delta^\beta} \xi^{\beta+2\alpha-1} D^\perp_\alpha(\xi,t_1) =
-A \beta \frac{H}{\delta^{1-2\alpha}} \left(\frac{\xi}{\delta} \right)^{\beta+2\alpha-1} D^\perp_\alpha(\xi,t_1).
\]
The latter expression does not exceed $-D^\perp_\alpha(\xi,t_1)$ provided that $C_1$ is chosen small enough.
\end{proof}

\noindent {\bf Acknowledgement.} \rm The author acknowledges support of the NSF grant DMS-0653813.
Part of this work was done while visiting St. Petersburg Institute of Information Technology, Mechanics and Optics, and the author
thanks Prof. Igor Popov for hospitality.

\end{document}